\documentclass[12pt]{amsart}
\usepackage{latexsym, amssymb, stmaryrd, eucal}

\newtheorem{thm}{Theorem}[section]

\newtheorem{cor}[thm]{Corollary}

\newtheorem{fact}[thm]{Fact}

\theoremstyle{definition}
\newtheorem{df}[thm]{Definition}

\newtheorem{convention}[thm]{Convention}

\newcommand{\BN}{\mathbb{N}}

\newcommand{\BL}{\mathbb{L}}
\newcommand{\BR}{\mathbb{R}}

\newcommand{\cu}[1]{\mathcal{#1}}

\newcommand{\Th}{\operatorname{Th}}

\makeatletter

\def\indsym#1#2{%
  \setbox0=\hbox{$\m@th#1x$}%
  \kern\wd0%
  \hbox to 0pt{\hss$\m@th#1\mid$\hbox to 0pt{$\m@th#1^{#2}$}\hss}%
  \lower.9\ht0\hbox to 0pt{\hss$\m@th#1\smile$\hss}%
  \kern\wd0}

\def\nindsym#1#2{%
  \setbox0=\hbox{$\m@th#1x$}%
  \kern\wd0%
  \hbox to 0pt{\hss$\m@th#1\not$\kern1.4\wd0\hss}
  \hbox to 0pt{\hss$\m@th#1\mid$\hbox to 0pt{$\m@th#1^{\,#2}$}\hss}%
  \lower.9\ht0\hbox to 0pt{\hss$\m@th#1\smile$\hss}%
  \kern\wd0}

\def\dotminussym#1#2{%
  \setbox0=\hbox{$\m@th#1-$}%
  \kern.5\wd0%
  \hbox to 0pt{\hss\hbox{$\m@th#1-$}\hss}%
  \raise.6\ht0\hbox to 0pt{\hss$\m@th#1.$\hss}%
  \kern.5\wd0}
\newcommand{\dotminus}{\mathbin{\mathpalette\dotminussym{}}}

\def\dotlesym#1#2{%
  \setbox0=\hbox{$\m@th#1-$}%
  \kern.5\wd0%
  \hbox to 0pt{\hss\hbox{$\m@th#1\le$}\hss}%
  \raise 1\ht0\hbox to 0pt{\hss$\m@th#1.$\hss}%
  \kern.5\wd0}
\newcommand{\dotle}{\mathbin{\mathpalette\dotlesym{}}}

\def\dotgesym#1#2{%
  \setbox0=\hbox{$\m@th#1-$}%
  \kern.5\wd0%
  \hbox to 0pt{\hss\hbox{$\m@th#1\ge$}\hss}%
  \raise 1\ht0\hbox to 0pt{\hss$\m@th#1.$\hss}%
  \kern.5\wd0}
\newcommand{\dotge}{\mathbin{\mathpalette\dotgesym{}}}

\begin{document}

\title{Continuous  Craig Interpolation}

\author{H. Jerome Keisler}

\address{University of Wisconsin-Madison, Department of Mathematics, Madison,  WI 53706-1388}
\email{keisler@math.wisc.edu}

\date{\today}


\begin{abstract}
We prove  analogues of the Craig interpolation theorem for the continuous model theory of metric structures.
 \end{abstract}

\maketitle


\section{Introduction}

In this paper we will prove two analogues of the Craig Interpolation Theorem for the continuous
model theory of metric structures
as developed in [BBHU] (2008).
The model theory of metric structures is currently an active area of research with
many applications to analysis, so it is important to clarify what happens to Craig interpolation in that setting.

In classical first order model theory, the Craig Interpolation Theorem ([Cr] 1957) says that:

\emph{For every sentence $\varphi$
in a vocabulary $V$ and $\psi$ in a vocabulary $W$ such that $\varphi\models\psi$, there is a sentence
$\theta$ in the common vocabulary $V\cap W$ such that $\varphi\models\theta$ and $\theta\models \psi$.}

The sentence $\theta$ is called a (Craig)  \emph{interpolant} of $\varphi$ and $\psi$.
Craig's proof used proof-theoretic methods.
A closely related result, the  Robinson Consistency Theorem ([Ro1] 1956), says that:

\emph{For every theory $T_V$ in a vocabulary $V$,
$T_W$ in a vocabulary $W$, and complete theory $T$ in the common vocabulary $V\cap W$, if each of
$T\cup T_V$ and $T\cup T_W$ is consistent then $T\cup T_V\cup T_W$ is consistent.}

Around 1959, several people noticed that the Craig Interpolation Theorem can be proved fairly easily from the Robinson Consistency Theorem, and vice versa
(see Feferman [Fe] (2008) and Robinson [Ro2] (1963), pp. 114--117).

This paper was partly motivated by ongoing work with my son Jeffrey M. Keisler.  In [KK1] (2012) and [KK2] (2014) we applied the first order Craig Interpolation Theorem
to questions arising from the field of decision analysis.  The present paper will enable the extension of that work from a discrete to a continuous setting.

In the introduction to the paper  [BYP] (2010), Ben Yaacov and Petersen wrote that ``continuous first-order logic satisfies a suitably phrased form of Craig's
interpolation theorem.'' (We will return to that point later).
The paper [BYP] did not prove or even state a form
of the Craig Interpolation Theorem for continuous logic.
However, the  paper [BYP] did develop a notion of formal proof for continuous logic
that could perhaps be used to  prove interpolation theorems.
A continuous analogue of Beth's Theorem was stated and proved in [BBHU] (Theorem 9.32).  See also the monograph [Fa] (2021)).

We assume the reader is familiar with [BBHU].
In order to state our main results we give a few reminders to fix notation.
 Instead of the equality predicate symbol, the logic in [BBHU] has a distinguished binary predicate
symbol $d$, called the \emph{distance predicate}.
We fix a set $L$ of finitary predicate and function symbols such that $d\in L$,
and a \emph{metric signature} $\BL$ that specifies a modulus of uniform continuity
with respect to $d$ for each predicate or function symbol in $L$.
A \emph{vocabulary} $V$ is a set $V\subseteq L$ such that $d\in V$.
Atomic formulas with vocabulary $V$ are the same as in first order logic.
$[0,1]$-\emph{valued structures} are like first order structures
except that the atomic formulas take values in $[0,1]$ instead of $\{\top,\bot\}$.

Continuous $V$-\emph{formulas} and $V$-\emph{sentences} are  built from atomic formulas
 with vocabulary $V$ using $\sup, \inf$ as quantifiers, and $n$-ary
 continuous  functions $C\colon[0,1]^n\to[0,1]$ (where $n\in\BN$) as connectives.
Each constant $r\in[0,1]$ also counts as a $V$-sentence (and a $0$-ary connective).
In a $[0,1]$-valued  structure $\cu M$ with vocabulary $V$, each $V$-sentence $\varphi$ has a truth value $\varphi^{\cu M}\in[0,1]$,
and $\varphi$ is called true in $\cu M$ if $\varphi^{\cu M}=0$.

A \emph{metric structure} is a $[0,1]$-valued structure $\cu M$ with a vocabulary $V\subseteq L$
such that the interpretation of $d$ in $\cu M$ is a complete metric, and $\cu M$
respects the bounds of uniform continuity in $\BL$.

\begin{convention}  \label{conv1} $V$ and $W$ will always denote vocabularies, $\varphi$ will always be a $V$-sentence, and
$\psi$ will always be a $W$-sentence.
\end{convention}

We will prove the following two analogues of
Craig Interpolation for metric structures.

\begin{df}  Let $\varepsilon\in(0,1]$.  A \emph{weak $\varepsilon$-interpolant} of $\varphi$ and $\psi$
is a $V\cap W$-sentence $\theta$ such that
$\varphi=0\models\theta=0$ and $\theta=0\models\psi\le\varepsilon$.
\end{df}

Intuitively,  $\theta$ is a sentence in the common language
such that $\theta$ is true whenever $\varphi$ is true, and $\psi$ is almost true whenever $\theta$ is true.

\begin{thm}   \label{t-weak}  (Weak Interpolant)  Suppose $\varphi=0\models\psi=0$.
Then for each $\varepsilon \in (0,1]$, $\varphi$ and $\psi$ have a weak $\varepsilon$-interpolant.
\end{thm}

\begin{df}  Let $\varepsilon\in(0,1]$.  A \emph{strong $\varepsilon$-interpolant} of $\varphi$ and $\psi$
is a $V\cap W$-sentence $\theta$ such that
$\models\varphi\ge \theta$ and $\models\theta\ge\psi - \varepsilon$.
\end{df}

\begin{thm} \label{t-strong}  (Strong Interpolant)  Suppose that
$\models \varphi\ge\psi.$
Then for each $\varepsilon\in (0,1]$, $\varphi$ and $\psi$ have a strong $\varepsilon$-interpolant.
\end{thm}

Theorems \ref{t-weak} and \ref{t-strong} can be compared with two interpolation theorems
for linear continuous logic in [Ba] (2014).  Linear continuous logic is like the continuous logic of [BBHU] except that
the only connectives are linear functions, and the space of  truth values is $\BR$ rather than $[0,1]$ (see [BM] (2023)).
The statements of Theorems \ref{t-weak} and \ref{t-strong} are similar to those of Propositions 6.7 and 6.8 in [Ba].
[Ba] gives direct proofs of both of Proposition 6.7 and 6.8 that resemble the model theoretic proof of the first order Craig interpolation theorem.

In the present setting of metric structures,
the proof of the Weak Interpolant Theorem \ref{t-weak} again resembles the model theoretic proof of the first order Craig interpolation theorem.
But the proof of the Strong Interpolant Theorem \ref{t-strong} is more difficult and
uses the Weak Interpolant Theorem.

Before going on, we note that in the other direction, the Weak Interpolant Theorem  is an easy consequence
of the Strong Interpolant Theorem, because $\models\varphi\ge\psi$ trivially implies $\varphi=0\models\psi=0$.
Thus every strong $\varepsilon$-interpolant of $\varphi$ and $\psi$
is a weak $\varepsilon$-interpolant of $\varphi$ and $\psi$.

The following corollary of the Weak Interpolant Theorem shows that  except for extreme cases, the notion of a strong
interpolant is strictly stronger that the notion of a weak interpolant.

\begin{cor}  \label{c-interp}
Suppose $\models\varphi\ge\psi$, $\varepsilon>0$, and there is a metric structure $\cu M$ with vocabulary  $V\cup W$
such that either $\varepsilon<\psi^{\cu M}$ or $\varphi^{\cu M}<1$.
Then there is a weak $\varepsilon$-interpolant of $\varphi$ and $\psi$
 that is not a strong $\varepsilon$-interpolant of $\varphi$ and $\psi$.
\end{cor}

\begin{proof}  By the Weak Interpolant Theorem, there is a weak $\varepsilon$-interpolant $\theta$ (of $\varphi$ and $\psi$).
If $\theta$ is already not a strong $\varepsilon$-interpolant, we are done.
Assume instead that $\theta$ is a strong $\varepsilon$-interpolant.  Let $C$ be a strictly increasing continuous
function from $[0,1]$ into $[0,1]$ such that $C(0)=0$.  Since $\theta$ is a weak $\varepsilon$-interpolant,
the sentence $C(\theta)$ is also a weak $\varepsilon$-interpolant.

Case 1:  $\varepsilon<\psi^{\cu M}$.   Since $\theta$ is a strong $\varepsilon$-interpolant,
$\theta^{\cu M}\ge\psi^{\cu M}-\varepsilon>0$.  Then we may take $C$ such that  $0<C(\theta^{\cu M})<\psi^{\cu M}-\varepsilon$,
so the sentence $C(\theta)$ is not a strong $\varepsilon$-interpolant.

Case 2:  $\varphi^{\cu M}<1$.  Since $\theta$ is a strong $\varepsilon$-interpolant, $1> \varphi^{\cu M}\ge\theta^{\cu M}$.
Then we may take $C$ such that $1>C(\theta^{\cu M})>\varphi^{\cu M}$, so again
the sentence $C(\theta)$ is not a strong $\varepsilon$-interpolant.
\end{proof}

It is instructive to compare weak and strong interpolants in the case that
$\models \varphi\ge\psi.$
A strong $\varepsilon$-interpolant  of $\varphi$ and $\psi$  is a $V\cap W$-sentence $\theta$
such that for each $r\in[0,1]$ and each $\cu M$ we have
$$\varphi^{\cu M}\le r\Rightarrow \theta^{\cu M}\le r,\quad\theta^{\cu M}\le r\Rightarrow (\psi^{\cu M}-\varepsilon)\le r .$$
Thus $\theta^{\cu M}$ is always between $\varphi^{\cu M}$ and $\psi^{\cu M}-\varepsilon$,
so ``almost'' between $\varphi^{\cu M}$ and $\psi^{\cu M}$.
For each  $r$, the Weak Interpolant Theorem for the sentences\footnote{$\varphi\dotminus r$ denotes the sentence $\max(\varphi-r,0)$.}
  $\varphi\dotminus r$ and $\psi\dotminus r$ gives us
a $V\cap W$-sentence $\theta_r$ (that may depend on $r$) such that for each $\cu M$,
$$\varphi^{\cu M}\le r\Rightarrow \theta_r^{\cu M}\le r,\quad\theta_r^{\cu M}\le r\Rightarrow (\psi^{\cu M}-\varepsilon)\le r .$$
To prove the Strong Interpolant Theorem, we will have to construct a sentence $\theta_r$ that does not depend on $r$.

This is illustrated in Figure 1.  The horizontal axis represents the class of all metric structures $\cu M$ with vocabulary $V\cup W$,
and the vertical axis represents the set of all $r\in[0,1]$.
The upper curve may be regarded as the ``graph'' of the function $\cu M\mapsto\varphi^{\cu M}$, and similarly for the other two curves.
\footnote{The picture is intended to illustrate the idea, but there are actually too many models to fit on a real line and the ``curves'' can be very wild.}

The region below the upper curve is the class of pairs $(\cu M,r)$ where $r\le\varphi^{\cu M}$, and the region below
the lower curve is the class of pairs $(\cu M,r)$ where $r\le(\psi^{\cu M}-\varepsilon)$.
The region below the middle curve is  the class of pairs $(\cu M,r)$ where $r\le\theta^{\cu M}$ for strong interpolants
and  $r\le\theta_r^{\cu M}$ for weak interpolants.

\begin{figure}
\begin{center}
\setlength{\unitlength}{1mm}
\begin{picture}(150,100)(0,10)
\put(10,15){\line(1,0){80}}
\put(10,90){\line(1,0){80}}
\put(10,15){\line(0,1){75}}
\put(90,15){\line(0,1){75}}

\qbezier(10,20),(50,20),(90,50)
\qbezier(10,30),(50,30),(90,70)
\qbezier(10,40),(50,45),(90,80)
\put(5,15){\makebox(0,0){$0$}}
\put(60,10){\makebox(0,0){$\cu M$}}
\put(5,60){\makebox(0,0){$r$}}

\put(5,90){\makebox(0,0){$1$}}
\put(110,50){\makebox(10,0)[r]{$(\psi^{\cu M}-\varepsilon)=r$}}
\put(110,70){\makebox(5,0)[c]{$\theta^{\cu M}=r$  or $\theta_r^{\cu M}=r $}}
\put(110,80){\makebox(0,0)[r]{$\varphi^{\cu M}=r$}}
\end{picture}
\end{center}
\caption{Interpolation when $\models \varphi\ge\psi.$}
\end{figure}
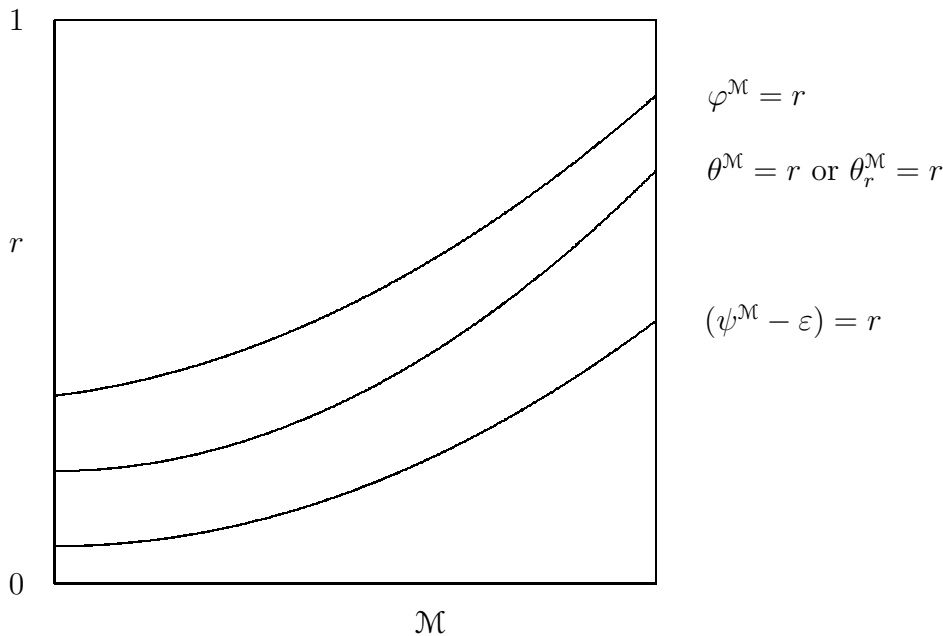

Our proof of the Strong Interpolant Theorem.will go roughly as follows.
Let $\varepsilon=2^{-n}$ for some  $n$.
Let $\mathbb D_n$ be the set of all multiples of $\varepsilon$ in $[0,1)$
(so $\mathbb D_n$ has cardinality $2^n$).
For each $r\in\mathbb D_n$, use the Weak Interpolant Theorem to get
the sentence $\theta_r$ defined in the above paragraphs.
 The hard part of the proof will be to
find a continuous function $C$ with $2^n$ variables  such that the sentence $C(\langle\theta_r\rangle_{r\in\mathbb D_n})$
is a strong $\varepsilon$-interpolant of $\varphi$ and $\psi$.

As mentioned above, the  paper [BYP] stated that
``continuous first-order logic satisfies a suitably phrased form of Craig's
interpolation theorem.''
When asked what a suitable form of the Craig Interpolation Theorem for continuous logic  would be,  Ben Yaacov [BY] (2022)
proposed the following, which  in the present setting is easily seen to be equivalent to the Weak Interpolant Theorem above
(and to Corollary \ref{c-separation} below).
\medskip

\emph{If $\{\varphi,\psi\}$ is inconsistent then there is a sentence $\theta$ in the common vocabulary such that $\{\varphi,\theta\dotminus 1/2\}$
is inconsistent and $\{\psi,1/2\dotminus \theta\}$ is inconsistent.}
\smallskip

Ben Yaacov [BY] pointed out that the above statement implies that a  uniform limit of sentences in the common vocabulary could serve as an interpolant.
This is also equivalent to the Weak Interpolant Theorem (see Corollary \ref{c-sequence} below).

In Section 2 we will prove an  analogue of the Robinson Consistency Theorem for metric structures
 (see Theorem \ref{t-robinson} below)
using results in the literature about saturated and special metric structures.
In Section 3 we will prove the Weak Interpolant Theorem from the continuous Robinson Consistency Theorem.
In Section 4 we will prove the Strong  Interpolant  Theorem from the Weak Interpolant Theorem,
using an argument that does not have a first order counterpart.

I thank Ita\"i Ben Yaacov and Jeffrey M. Keisler for helpful discussions related to this paper.

\section{Robinson Consistency Theorem}  \label{s-prelim}

We first recall some basic notation and facts  about metric structures that will be needed in this and the following sections.  We will then
prove a continuous analogue of the Robinson Consistency Theorem.

By a $V$-\emph{theory} we mean a set of $V$-sentences.  We say that $\cu M$ is a \emph{metric model} of $T$,
in symbols $\cu M\models T$,  if $\cu M$ is a metric structure, $T$ is a $V$-theory where $V$ is the vocabulary of $\cu M$,
 and every $\theta\in T$ is true (has truth value $0$) in $\cu M$.
For  $V$-theories $T$ and $U$,  $T\models U$ means that every metric model of $T$ is a metric model of $U$.
We write $T\models\varphi\ge\theta$ if $\varphi^{\cu M}\ge\theta^{\cu M}$ for every metric model $\cu M\models T$  (similarly for $\le$ and $=$).
Thus $T\models \{\varphi\}$ if and only if $T\models \varphi=0$.
Note that to the right of the $\models$ symbol we can have either a set of sentences or an inequality or equation between a pair of sentences.

$T$ is \emph{consistent} if $T$ has at least one metric model, and
$T$ is \emph{inconsistent} if $T$ has no metric models.
If $\cu M$ is a metric structure, or even just a $[0,1]$-valued structure, with vocabulary $V$,
the \emph{theory of} $\cu M$ is the set $\Th(\cu M)$ of all $V$-sentences true in $\cu M$.

Let   $\cu N$ be a metric structure with vocabulary $V\cup W$.
The $V$-\emph{part} of $\cu N$ is the  metric structure $\cu M=\cu N\upharpoonright V$ with vocabulary $V$  that agrees with $\cu N$ on all symbols of $V$,
and $\cu N$ is called an \emph{expansion} of $\cu M$ to $V\cup W$.  Note that if $T$ is a $V$-theory, then $T$ is also a $V\cup W$-theory,
and that $\cu N\models T$  if and only if $\cu N\upharpoonright V\models T$ .

For  $V$-sentences $\varphi$ and $\theta$,  $\varphi\dotle\theta$ (also written $\varphi\dotminus\theta$) denotes the $V$-sentence
$\max(\varphi-\theta,0)$.  The sentence $\varphi\dotle\theta$ is useful because in every metric structure $\cu M$ we have
$$(\varphi\dotle\theta)^{\cu M}=0 \mbox{ if and only if } \varphi^{\cu M}\le\theta^{\cu M}.$$
We will use either $\dotle$ or $\dotminus$, whichever seems more natural in each case.
Also,  $\varphi\dotge\theta$ denotes the $V$-sentence $\theta\dotle\varphi$, and
$\varphi\dotplus \theta$ denotes the $V$-sentence $\min(\varphi +\theta,1)$.

We will repeatedly use the following fact and its corollary.

\begin{fact}  \label{f-compactness}  (Compactness Theorem, by Theorem 5.8 of [BBHU])
If every finite subset of a $V$-theory  $T$ is consistent, then $T$ is consistent.
\end{fact}

\begin{cor}  \label{c-compact}  (Compactness Corollary)
$T\cup\{\varphi\}$ is inconsistent if and only if   there exists $r\in(0,1]$ such that
 $T\models r\le \varphi$.
\end{cor}

\begin{proof}  If there exists $r\in(0,1]$ such that $T\models r\le\varphi$, it trivially follows that $T\cup\{\varphi\}$ has no metric models
and hence is inconsistent.

Suppose there is no $r\in(0,1]$ such that $T\models r\le\varphi$.
Then for each $r\in(0,1]$ there is a metric model $\cu M$ of $T$ such that $(r\dotle\varphi)^{\cu M}>0$,
so $r>\varphi^{\cu M}$,  $\varphi^{\cu M}\le r$, and $(\varphi\dotle r)^{\cu M}=0$.
Then every finite subset of the theory $U=T\cup\{\varphi\dotle r\colon r\in(0,1]\}$ is consistent.
By the Compactness Theorem, $U$ is consistent, and thus has a metric model $\cu N$.
Then $(\varphi\dotle r)^{\cu N}=0$ for each $r\in(0,1]$, so $\varphi^{\cu N}=0$ and $\cu N\models T\cup\{\varphi\}$.
Therefore $T\cup\{\varphi\}$ is not inconsistent.
\end{proof}

Two metric structures $\cu M, \cu N$ with vocabulary $V$ are \emph{isomorphic}, in symbols $\cu M\cong\cu N$, if there is a bijection
from the universe of $\cu M$ onto the universe of $\cu N$
that preserves the truth value of all atomic formulas.

\begin{fact}  If $\cu M\cong\cu N$ then $\varphi^{\cu M}=\varphi^{\cu N}$ for every $V$-sentence $\varphi$.
\end{fact}

A cardinal $\kappa$ is \emph{special} if $2^\lambda\le\kappa$ for all $\lambda\le\kappa$
(for example, $\beth_\omega$ is special).
A metric structure $\cu M$ is $\kappa$-\emph{special} if $\kappa$ is an uncountable special cardinal, $|M|\le\kappa$, and
$\cu M$ is the union of an elementary chain of metric structures $\langle \cu M_\lambda\colon\aleph_0\le\lambda<\kappa\rangle$
such that each $\cu M_\lambda$  is $\lambda^+$-saturated.

\begin{fact}  \label{f-special}  (Facts 2.4.6 and 2.4.8 in [Ke].)  Suppose $T$ is a $V$-theory, $\kappa$ is special,
and $\kappa\ge|V|+\aleph_0$.
If $T$ is consistent then $T$ has a $\kappa$-special model with vocabulary $V$.
If $T$ is complete  then up to isomorphism
$T$ has a unique $\kappa$-special metric model with vocabulary $V$.
\end{fact}

\begin{fact}  \label{f-reduct}  (Remark 2.4.4 in [Ke])
The $V$-part of a $\kappa$-special metric model with vocabulary $V\cup W$ is  a $\kappa$-special metric model with vocabulary $V$.
\end{fact}

\begin{convention}  \label{conv2}  As before, $V$ and $W$ will always denote vocabularies, $\varphi$ will always be a $V$-sentence, and
$\psi$ will always be a $W$-sentence.  We also fix a $V$-theory $T_V$ and a $W$-theory $T_W$.
\end{convention}

The following result is a continuous analogue of the Robinson Consistency Theorem.

\begin{thm}  \label{t-robinson}  Let  $T=\Th(\cu P)$ for some metric structure  $\cu P$ with vocabulary $V\cap W$.
If $T\cup T_V$ is consistent and $T\cup T_W$ is consistent, then $T\cup T_V\cup T_W$ is consistent.
\end{thm}

\begin{proof}  Let $\kappa$ be a special cardinal with  $\kappa >|V\cup W|+\aleph_0$.  By Fact \ref{f-special},   there is a $\kappa$-special
metric model $\cu M$ of $T\cup T_V$ with vocabulary $V$, and a $\kappa$-special metric
model $\cu N$ of $T\cup T_W$ with vocabulary $W$.  By Fact \ref{f-reduct},
the $V\cap W$-parts of $\cu M$ and of $\cu N$ are both metric models of $T$, and are both $\kappa$-special  with vocabulary $V\cap W$.
Since $T$ is complete, by Fact \ref{f-special}
 the $V\cap W$-parts of $\cu M$ and $\cu N$ are isomorphic.
Therefore we may take $\cu M$ and $\cu N$ to have the same $V\cap W$-parts.
Then $\cu M$ and $\cu N$ have a common expansion to a metric model of $T\cup T_V\cup T_W$.
\end{proof}

\section{Weak Interpolants} \label{s-weak}

Convention \ref{conv2}, where we fix the theories $T_V$ and $T_W$, is still in force.
Instead of proving the interpolation theorems in the simple form stated in the Introduction, it will be easier and hence better
to prove them in the more general setting where we restrict attention to the class of metric models of $T_V\cup T_W$
instead of the class of all metric structures.

The  Weak Interpolant Theorem \ref{t-weak} in the Introduction is a special case of Theorem \ref{t-weak2}
below (that special case arises when both $T_V$ and $T_W$ are  empty sets of sentences).

\begin{thm}   \label{t-weak2}  Suppose that
$$T_V\cup T_W\cup\{\varphi\}\models \psi=0.$$
Then for each $\varepsilon \in (0,1]$,
there is a $V\cap W$-sentence $\theta$ such that
$$T_V\cup\{\varphi\}\models \theta=0,\quad T_W\cup \{\theta\}\models \psi\le\varepsilon.$$
\end{thm}

\begin{proof}
Let $U$ be the set of all $V\cap W$-sentences $\rho$ such that $T_V\cup\{\varphi\}\models\rho=0$.  Then every metric model of $T_V\cup \{\varphi\}$
is a metric model of $U$. We first prove the following Claim.
\medskip

\textbf{Claim.}  $T_W\cup U\models\psi=0$.

\begin{proof} [Proof of Claim]

Suppose $\cu M$ is a metric model of $T_W\cup U$.
To show that $\psi^{\cu M}=0$, we assume that $0<\psi^{\cu M}$ and get a contradiction.
Since $0<\psi^{\cu M}$, $r\le \psi^{\cu M}$ for some $r\in(0,1]$.
 Let $\cu P$ be the $V\cap W$-part of $\cu M$ and $T=\Th(\cu P)$.
Then $\cu M$ is a metric model of $T\cup T_W\cup\{r\dotle\psi\}$, so
\begin{equation}  \label{eq-main1}
T\cup T_W\cup\{r\dotle\psi\} \mbox{ is consistent}.
\end{equation}
We prove
\begin{equation} \label{eq-main2}
T\cup T_V\cup \{\varphi\} \mbox{  is consistent}.
\end{equation}
Assume that (\ref{eq-main2}) fails.
By the Compactness Theorem,  there is
a finite subset $T_0\subseteq T$ such that
$T_0\cup T_V\cup\{\varphi\}$  is inconsistent.
Let $\rho$ be the $V\cap W$-sentence $\max(T_0)$.  Then
$\{\rho\}\cup T_V\cup\{\varphi\}$ is inconsistent.
By the Compactness Corollary, there is an $s\in(0,1]$ such that
$$T_V\cup\{\varphi\}\models s\le\rho.$$
Hence $s\dotle\rho\in U$.
We have $\cu M\models\rho=0$ because $\cu M\models T$, and $\cu M\models s\le\rho$ because $\cu M\models U$.
This is a contradiction, so (\ref{eq-main2}) holds after all.

It follows from (\ref{eq-main1}), (\ref{eq-main2}), and Theorem \ref{t-robinson} that
$$ T\cup T_V\cup T_W\cup \{\varphi, r\dotle\psi\} \mbox{ is consistent}.$$
This contradicts the hypothesis that $T_V\cup T_W\cup\{\varphi\}\models\psi=0$
in the statement of the theorem, so the assumption that $0<\psi^{\cu M}$ is false.
Thus $\psi^{\cu M }=0$ and hence
$$T_W\cup U\models\psi=0$$
as required.  This completes the proof of the Claim.
\end{proof}

 Let $\varepsilon\in (0,1]$.
Then by the Claim,
$$T_W\cup U\cup \{\varepsilon\dotle \psi\}$$
is inconsistent.
By the Compactness Theorem, there is a  finite subset $U^\varepsilon$ of $U$
such that
$$T_W\cup U^\varepsilon\cup \{\varepsilon\dotle \psi\}$$
is inconsistent.
Let $\theta:=\max(U^\varepsilon)$.  Then $\theta\in U$, so
$$T_V\cup T_W\cup\{\varphi\}\models \theta=0,\quad T_V\cup T_W\cup \{\theta\}\models \psi\le\varepsilon.$$
as required.
\end{proof}

\begin{cor}   \label{c-weak2}  Suppose that
$$T_V\cup T_W\cup\{\varphi\}\models \psi=0.$$
Then for each $\varepsilon \in (0,1]$,
there is a $V\cap W$-sentence $\rho$ such that
$$T_V\cup\{\varphi\}\models \rho\le\varepsilon,\quad T_W\cup \{\rho\dotle\varepsilon\}\models \psi\le\varepsilon.$$
\end{cor}

\begin{proof} If $\theta$ is as in Theorem \ref{t-weak2} then $\rho :=\theta\dotplus\varepsilon$ has the
required properties.
\end{proof}

\begin{cor} \label{c-separation}
Suppose $T_V\cup T_W\cup\{\varphi,\psi\}$ is inconsistent.
Then there is a $V\cap W$-sentence $\theta$  such that
$$T_V\cup\{\varphi\}\models \theta=1,\qquad T_W\cup\{\psi\}\models \theta=0.$$
\end{cor}

\begin{proof}  By the  Compactness Corollary, there exists $r\in(0,1]$ such that
$$T_V\cup T_W\cup\{\varphi\}\models r\le \psi.$$
Let $\varepsilon\in(0,r]$.  By Theorem \ref{t-weak2}
there is a $V\cap W$-sentence $\rho$ such that
$$T_V\cup\{\varphi\}\models \rho=0,\qquad T_W\cup\{\rho\}\models (r\dotminus \psi)\le \varepsilon/2.$$
One can easily check that for any $x\in[0,1]$ we have $ (r\dotminus x)\le\varepsilon/2$ if and only if $(r-\varepsilon/2)\le x$.
Therefore
$$0< r-\varepsilon/2,\qquad T_W\cup\{\rho\}\models r-\varepsilon/2\le\psi,$$
so $T_W\cup\{\rho\}\cup \{\psi\}$ is inconsistent.
By the  Compactness Corollary,
there exists $s\in(0,1/2]$ such that
$$T_W\cup\{\psi\}\models s\le\rho.$$
Then  $\theta:= 1\dotminus(\rho /s)$
has the required properties.
\end{proof}

Note that Theorem \ref{t-weak2} is also an easy consequence of Corollary \ref{c-separation}.
To see that, suppose $T_V\cup T_W\cup\{\varphi\}\models\psi=0$ and $\varepsilon\in(0,1]$.  Then $T_V\cup T_W\cup \{\varphi,\varepsilon\dotle\psi\}$
is inconsistent, so Corollary \ref{c-separation} gives a $V\cap W$-sentence $\theta$ such
that $T_V\cup\{\varphi\}\models\theta=0$ and $T_W\cup\{\varepsilon\dotle\psi\}\models\theta=1$,
and hence $T_W\cup\{\theta\}\models \psi\le\varepsilon$.

While the Weak Interpolant Theorem shows that a single $V\cap W$ sentence serves as a weak $\varepsilon$-interpolant,
the next corollary shows that a countable set of $V\cap W$-sentences can serve as a weak interpolant (without the $\varepsilon$).

\begin{cor}  \label{c-countable}
Suppose   $T_V\cup T_W\cup\{\varphi\}\models\psi=0$.  Then
for some countable set $\Theta$ of $V\cap W$-sentences,
$T_V\cup\{\varphi\}\models\Theta$  and  $T_W\cup\Theta\models\psi=0$.
\end{cor}

\begin{proof}
By Theorem \ref{t-weak2},  for each $n\in\BN$ there is a $V\cap W$-sentence $\theta_n$
such that  $T_V\cup\{\varphi\}\models \rho_n=0$ and $T_W\cup\{\rho_n\}\models\psi\le 2^{-n}$.
So the result holds with
$\Theta=\{\rho_n\mid n\in\BN\}.$
\end{proof}

We say that a sequence of sentences $\langle\theta_n\rangle_{n\in\BN}$ is \emph{uniformly convergent}
if for every $\varepsilon>0$ there exists $n$ such that for all $k>n$, $\models|\theta_k-\theta_n|\le\varepsilon$.
Note that if  $\langle\theta_n\rangle_{n\in\BN}$ is uniformly convergent then $\lim_{n\to\infty} \theta_n$
exists in every metric structure, and the ``rate of convergence'' is uniform across all metric structures.
Also, if $\models|\theta_{n+1}-\theta_n|\le 2^{-n}$ for every $n$ then
$\langle\theta_n\rangle_{n\in\BN}$ is uniformly convergent.

The next corollary shows that a uniform limit of $V\cap W$ sentences can serve as a weak interpolant.

\begin{cor}  \label{c-sequence}
Suppose  $T_V\cup T_W\cup\{\varphi\}\models\psi=0.$
Then there is a uniformly convergent sequence $\langle \theta_n\rangle_{n\in\BN}$ of $V\cap W$-sentences such that
\begin{itemize}
\item[(i)] In every model  of $T_V$, if $\varphi=0$ then $\lim_{n\to\infty} \theta_n=0$.
\item[(ii)]  In every model  of $T_W$, if $\lim_{n\to\infty}\theta_n=0$ then  $\psi=0$.
\end{itemize}
\end{cor}

\begin{proof}
Let $\theta_n=\max_{m\le n}\min(\rho_m,2^{-m})$ where $\rho_n$ is as  in the proof of
Corollary \ref{c-countable}.  For each $n$, $\theta_n$ is a $V\cap W$-sentence.
We let $\cu M$ be a metric structure with vocabulary $V\cap W$.
Clearly, $\theta_n^{\cu M}\le\theta_{n+1}^{\cu M}$ for each $n.$
Therefore $\lim_{n\to\infty} \theta_n^{\cu M}=0$ if and only if $(\forall n)\theta_n^{\cu M}=0$.
We show  that
\begin{equation}  \label{eq3}
\theta_{n+1}^{\cu M}\le\theta_n^{\cu M}+2^{-(n+1)}
\end{equation}
for each $n$.  (\ref{eq3}) is trivially true if $\theta_{n+1}^{\cu M}=\theta_n^{\cu M}$.
If $\theta_{n+1}^{\cu M}>\theta_n^{\cu M}$, then (\ref{eq3}) still holds because
$$\theta_{n+1}^{\cu M}=\min(\rho_{n+1}^{\cu M},2^{-(n+1)})\le 2^{-(n+1)}\le\theta_n^{\cu M}+2^{-(n+1)}.$$
It follows that $\langle\theta_n\rangle_{n\in\BN}$ is  uniformly convergent,
and by Corollary \ref{c-countable}, (i) and (ii) hold.
\end{proof}

\section{Strong Interpolants}  \label{s-strong}

Convention \ref{conv2} is still in force.
The  Strong Interpolant Theorem \ref{t-strong}  in the Introduction is a special case of  Theorem \ref{t-strong2} below
(Theorem \ref{t-strong2} reduces to Theorem \ref{t-strong} when
both $T_V$ and $T_W$ are  empty sets of sentences).

\begin{thm}  \label{t-strong2}  Suppose that  $T_V\cup T_W\models\varphi\ge\psi$.  Then for each $\varepsilon>0$ there is a $V\cap W$-sentence
$\theta$ such that
\begin{itemize}
\item[(i)] $T_V\models\varphi\ge\theta$,
\item[(ii)] $ T_W\models \theta\ge(\psi\dotminus \varepsilon)$.
\end{itemize}
\end{thm}

\begin{proof}
Fix $n\in\BN$ and let $\varepsilon=2^{-n}$.  It suffices to  find a $V\cap W$-sentence $\theta$ such that
\begin{itemize}
\item[(i*)] $T_V\models\varphi\ge(\theta\dotminus\varepsilon)$.
\item[(ii)] $ T_W\models \theta\ge(\psi\dotminus \varepsilon)$.
\end{itemize}
Because if (i*) and (ii) hold with $\varepsilon/2$ in place of $\varepsilon$,
then (i) and (ii) hold with $\theta\dotminus\varepsilon/2$ in place of $\theta$.

Our proof will have two parts.  In the first part we will use Theorem \ref{t-weak2}
to find, for each $k<2^n$, a $V\cap W$-sentence $\rho_k$ such that

\begin{equation} \label{eq4}
T_V\cup\{\varphi\dotle k\varepsilon\}\models \rho_k=0,
\end{equation}
and
\begin{equation} \label{eq5}
T_W\cup\{\psi\dotge (k+1)\varepsilon\}\models \rho_k=1.
\end{equation}

In the second part of the proof we will find a continuous function
$f\colon[0,1]^{2^n}\to[0,1]$ such that $\theta:=f(\langle\rho_k\rangle_{k<2^n})$
satisfies (i*) and (ii).

The ``graph'' of  $\rho_k^{\cu M}$ as a function of $\cu M$ is illustrated by the bold line in Figure 2.
As in Figure 1, we put the interval $[0,1]$ on the vertical axis and the class of all metric models $\cu M$ of $T_V\cup T_W$
on the horizontal axis.

\begin{figure}
\begin{center}
\setlength{\unitlength}{1mm}
\begin{picture}(150,100)(0,10)

\put(10,15){\line(1,0){80}}
\put(10,90){\line(1,0){80}}
\put(10,15){\line(0,1){75}}
\put(90,15){\line(0,1){75}}

\qbezier(10,30),(50,30),(90,70)
\qbezier(10,40),(50,45),(90,80)
\put(41.2,48.5){\circle*{1}}
\put(51,40.5){\circle*{1}}
\put(35,52){\makebox(0,0){$(\varphi^{\cu M},k\varepsilon)$}}
\put(62,38){\makebox(0,0){$(\psi^{\cu M},k\varepsilon+\varepsilon)$}},
\put(40,15){\line(0,1){33}}
\put(50,40.5){\line(0,1){49}}

\put(5,15){\makebox(0,0){$0$}}
\put(60,10){\makebox(0,0){$\cu M$}}
\put(5,60){\makebox(0,0){$r$}}
\put(5,90){\makebox(0,0){$1$}}
\put(110,70){\makebox(0,0)[r]{$\psi^{\cu M}=r$}}
\put(110,80){\makebox(0,0)[r]{$\varphi^{\cu M}=r$}}

\thicklines
\put(10,15){\line(1,0){30} }
\put(50,90){\line(1,0){40}}
\qbezier(40,15),(44,44),(50,90)
\put(38,70){\makebox(0,0){$\rho_k^{\cu M}=r$}}

\end{picture}
\end{center}
\caption{Graphs of $\varphi^{\cu M}, \psi^{\cu M}$, and $\rho_k^{\cu M}$}
\end{figure}
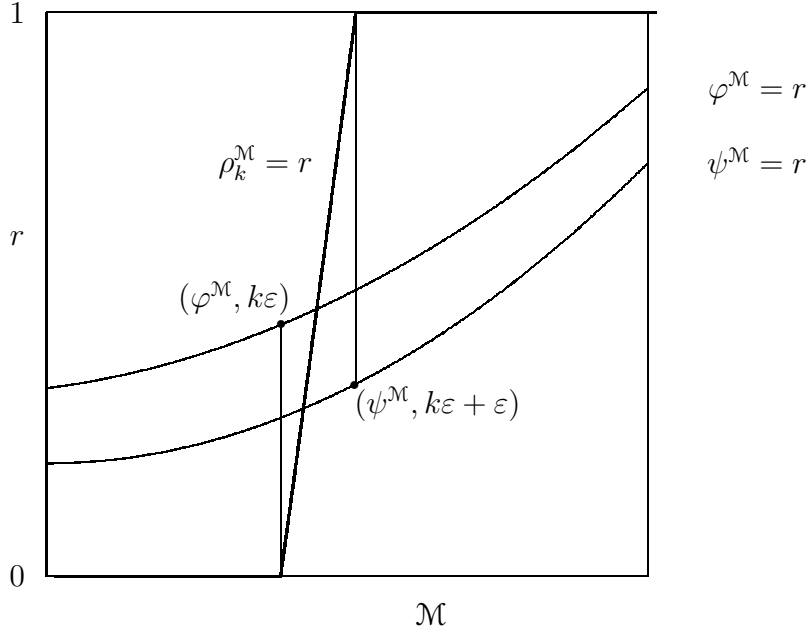

\emph{First part of proof:}
Let $k<2^n$.  Since $T_V\cup T_W\models \varphi\dotge\psi$, we have
$$T_V\cup T_W\cup\{\varphi\dotle k\varepsilon\}\models \psi\le k\varepsilon.$$
Then by Theorem \ref{t-weak2} (with $\varphi\dotle k\varepsilon,\psi\dotle k\varepsilon,\varepsilon/2$ in place of
$\varphi,\psi,\varepsilon$),    there is
a $V\cap W$-sentence $\gamma_k$ such that
\begin{equation}  \label{eq6}
T_V\cup\{\varphi\dotle k\varepsilon\}\models \gamma_k=0
\end{equation}
and
$$T_W\cup\{\gamma_k\}\models \psi\le k\varepsilon + \varepsilon/2.$$
Therefore the set of sentences
$$ T_W\cup \{\gamma_k\}\cup\{ \psi\dotge (k+1)\varepsilon\}$$
is inconsistent.  By  Compactness Corollary, there exists $r\in(0,1]$ such that
\begin{equation} \label{eq7}
T_W\cup\{\psi\dotge (k+1)\varepsilon\}\models \gamma_k\ge r.
\end{equation}
Now  $\rho_k:= \min(\gamma_k/r,1)$
 is a  $V\cap W$-sentence.
By the definition of $\rho_k$, (\ref{eq4}) follows at once from (\ref{eq6}), and (\ref{eq5}) follows at once from (\ref{eq7}).

\emph{Second part of proof:}
Let $f\colon[0,1]^{2^n}\to[0,1]$ be the continuous function
$$f(\vec x):=\max_{k<2^n}\left[ ( k+1)\varepsilon\prod_{j\le k}x_j\right].$$
Let
$$\theta:=f(\langle \rho_k\rangle_{k<2^n}).$$
Since $f$ is continuous and each $\rho_k$ is a $V\cap W$-sentence, $\theta$
is a $V\cap W$-sentence.
We show that $\theta$ satisfies (i*) and (ii).
It is clear that $f$ has the following properties for each $k<2^n$:
\begin{itemize}
\item[(a)]  If $x_k=0$ then $f(\vec x)\le k\varepsilon$.
\item[(b)]  If $x_j=1$ for each $j\le k$ then $f(\vec x)\ge (k+1)\varepsilon$.
\end{itemize}

To prove (i*), we work in an arbitrary metric  model $\cu M$ of $T_V$.
In $\cu M$, if $\varphi\ge 1-\varepsilon$ then it is trivial that $ \varphi\ge(\theta\dotminus\varepsilon)$.
 Suppose instead that $\varphi<1-\varepsilon$.   Then there is a
unique $k<2^n-1$ such that
$$k\varepsilon\le\varphi< (k+1)\varepsilon.$$
Then $k+1<2^n$ and $\varphi\le(k+1)\varepsilon$, so by
(\ref{eq4}) we have $\rho_{k+1}=0$.
Hence  by (a), $\theta\le (k+1)\varepsilon$.  Therefore $(\theta\dotminus\varepsilon)\le k\varepsilon$
and $k\varepsilon\le\varphi$.  This proves (i*).

To prove (ii) we work in an arbitrary model of $T_W$.  If $\psi\le\varepsilon$ then it is trivial that $\theta\ge(\psi\dotminus\varepsilon)$.
Suppose instead that $\psi>\varepsilon$.  Then there is a unique $0<\ell<2^n$ such that
$$\ell\varepsilon<\psi\le(\ell+1)\varepsilon.$$
Hence for each $j\le \ell-1$ we have $(j+1)\varepsilon<\psi$, so $(j+1)\varepsilon\le\psi$.
By (\ref{eq5}),  $\rho_j=1$ for each $j\le\ell-1$.   Since $\ell-1\in\BN$ we may apply (b) to get
$\theta\ge \ell\varepsilon$.  Since $\psi\le(\ell+1)\varepsilon$, we also have $\ell\varepsilon\ge\psi\dotminus\varepsilon $.
Therefore $\theta\ge(\psi\dotminus\varepsilon)$.  This proves (ii).
\end{proof}

Note that in the above proof, the function $f$ is non-decreasing in each variable.  Such functions are sometimes called aggregation functions.

\begin{cor}   Suppose that  $T_V\cup T_W\models\varphi\ge\psi$.  Then for each $\varepsilon>0$ there is a $V\cap W$-sentence
$\rho$ such that
\begin{itemize}
\item[(i)] $T_V\models\varphi\ge(\rho\dotminus\varepsilon)$,
\item[(ii)] $ T_W\models (\rho\dotminus\varepsilon)\ge(\psi\dotminus \varepsilon)$.
\end{itemize}
\end{cor}

\begin{proof} If $\theta$ is as in Theorem \ref{t-strong2} then $\rho :=\theta\dotplus\varepsilon$ has the
required properties.
\end{proof}

The next corollary shows that a uniformly convergent sequence of $V\cap W$-sentences can serve as a strong interpolant.

\begin{cor}  \label{c-limit}
Suppose  $T_V\cup T_W\models\varphi\dotge\psi$.
Then there is a uniformly convergent sequence $\langle \theta_n\rangle_{n\in\BN}$ of $V\cap W$-sentences such that
\begin{itemize}
\item[(i)] In every metric model  of $T_V$, $\varphi\ge\lim_{n\to\infty} \theta_n$.
\item[(ii)]  In every metric model  of $T_W$, $\lim_{n\to\infty}\theta_n\ge \psi$.
\end{itemize}
\end{cor}

\begin{proof}
By Theorem \ref{t-strong2}, for each $n\in\BN$ there is a $V\cap W$-sentence $\gamma_n$ such that
$\varphi\ge \gamma_n$ in any metric model of $T_V$,
and $\gamma_n\ge(\psi\dotminus 2^{-n})$
in any metric model of $T_W$.
Now let $\theta_0:=0$, and for each $n$ let
$$\theta_{n+1}:=\max(\theta_{n},\min(\gamma_{n},\theta_{n}\dotplus 2^{-n})).$$
It is clear that each $\theta_n$ is a $V\cap W$-sentence, and
$$\models \theta_{n}\le\theta_{n+1},\quad
\models |\theta_{n+1}-\theta_{n}|\le 2^{-n}.$$
Therefore $\langle\theta_n\rangle_{n\in\BN}$
is uniformly convergent.
We also have
$$\models\theta_{n+1}\le \max(\theta_n,\gamma_n),\quad T_V\models \varphi\ge\gamma_n.$$
It follows by induction that $T_V\models\varphi\ge\theta_n$, so (i) holds.
And
$$ \models\theta_{n+1}\ge\min(\gamma_n,\theta_n\dotplus 2^{-n}),\quad T_W\models \gamma_n\ge(\psi\dotminus 2^{-n}),$$
so it follows by induction that $T_W\models\theta_n\ge(\psi\dotminus 2^{-n})$.  Hence (ii) holds as well.
\end{proof}

The arguments in this paper actually prove more general results that apply to
arbitrary $[0,1]$-valued structures as developed in [Ke] as well as to metric structures.

Hereafter, $V$ and $W$ will  denote vocabularies, $\varphi$ will always be a $V$-sentence, and
$\psi$ will always be a $W$-sentence.  We fix a $V$-theory $T_V$ and a $W$-theory $T_W$.  But no metric signature is given.
$T\models_g U$ will  mean that every general  $[0,1]$-valued model of $T$ is a  model of $U$.

Here is the analogue of the Weak Interpolant Theorem for general structures.

\begin{thm}   \label{t-weak3}  Suppose that
$$T_V\cup T_W\cup\{\varphi\}\models_g \psi=0.$$
Then for each $\varepsilon \in (0,1]$,
there is a $V\cap W$-sentence $\theta$ such that
$$T_V\cup\{\varphi\}\models_g \theta=0,\quad T_W\cup \{\theta\}\models_g \psi\le\varepsilon.$$
\end{thm}

Here is the analogue of the Strong Interpolant Theorem for general structures.

\begin{thm}  \label{t-strong3}  Suppose that  $T_V\cup T_W\models_g\varphi\ge\psi$.  Then for each $\varepsilon>0$ there is a $V\cap W$-sentence
$\theta$ such that
\begin{itemize}
\item[(i)] $T_V\models_g\varphi\dotge\theta$,
\item[(ii)] $ T_W\models_g \theta\dotge(\psi\dotminus \varepsilon)$.
\end{itemize}
\end{thm}

To complete the picture, we recall one more result from [BBHU].

\begin{fact}  \label{f-completion}  (Theorem 3.7 in  [BBHU])  If $\BL$ is a metric signature with vocabulary $V$
and $T_V$ contains sentences that specify  the uniform continuity moduli for $\BL$,
then every general $[0,1[$-valued model of $T_V$ is elementarily
equivalent to a metric structure with signature $\BL$.
\end{fact}

Thus when $T_V$ and $T_W$ contain  sentences that specify  the uniform continuity moduli for $\BL$,
 Theorem \ref{t-weak2} is a special case of Theorem \ref{t-weak3}, and
 Theorem \ref{t-strong2} is a special case of Theorem \ref{t-strong3}.

\end{document}